\documentclass[12pt,leqno]{article}
\usepackage{amsmath,amssymb,amsbsy,amsfonts,amsthm,latexsym,
         amsopn,amstext,amsxtra,euscript,amscd,amsthm,mathdots}

\allowdisplaybreaks

\numberwithin{equation}{section}

\newtheorem{lem}{Lemma}

\newtheorem{thm}{Theorem}

\newtheorem{rem}{Remark}

\begin{document}

\begin{large}
\centerline{\Large \bf Polynomials with integer roots}
\end{large}
\vskip 10pt
\begin{large}
\centerline{\sc  Patrick Letendre}
\end{large}
\vskip 10pt
\begin{abstract}
Let $\mathcal{F}_n$ be the set of unitary polynomials of degree $n \ge 2$ that have their roots in $\mathbb{Z}^*$. We note
$$
Q(x) := x^n+a_{1}x^{n-1}+\dots+a_{n}.
$$
We show that any two fixed consecutive coefficients $(a_{j},a_{j+1})$ ($j \in \{1,\dots,n-1\})$ define finitely many polynomials of $\mathcal{F}_n$.
\end{abstract}



\vskip 20pt

\section{Introduction and notation}

Let $\mathcal{F}_n$ be the set of unitary polynomials of degree $n \ge 2$ that have their roots in $\mathbb{Z}^*$. For $Q(x) \in \mathcal{F}_n$, we write
\begin{eqnarray}
Q(x) & := & a_{0}x^n + a_{1} x^{n-1} + \cdots + a_{n-1} x + a_{n}\\
& = & \prod_{j=1}^{n}(x-x_j)\\
& = & x^n - s_1 x^{n-1} + \cdots + (-1)^{n-1}s_{n-1} x + (-1)^n s_n
\end{eqnarray}
where $s_j=s_j(x_1,\dots,x_n)$ are the elementary symmetric functions, that is
\begin{eqnarray*}
s_1 & = & \sum_{1 \le i \le n}x_i \\
s_2 & = & \sum_{1 \le i < j \le n}x_i x_j \\
& \cdots & \\
s_n & = & x_1 \cdots x_n.
\end{eqnarray*}
Given any fixed set of coefficients, we are interested by the polynomials of $\mathcal{F}_n$ that have these coefficients.

Let us illustrate the problem with two simple examples. Let $Q(x) \in \mathcal{F}_n$ and assume that $a_n$ is known. Since $(-1)^{n}a_{n} = s_{n} = x_{1}\cdots x_{n}$, we deduce that there are finitely many such polynomials.

Now, let $Q(x) \in \mathcal{F}_n$ and assume that $a_1$ and $a_2$ are known. From the identity
$$
a_1^2-2a_2 = s_1^2-2s_2 = x_1^2 + \cdots + x_n^2,
$$
it is clear that there are finitely many such polynomials.

The main result is contained in Theorem \ref{thm:1}.

\begin{thm}\label{thm:1}
Let $Q(x) \in \mathcal{F}_n$ be a polynomial satisfying $\max(|a_{j}|,|a_{j+1}|) \le M$ for some $j \in \{1,\dots,n-1\}$. Then
\begin{equation}\label{ineq_thm}
\max(|x_1|,\dots,|x_n|) \le \alpha_n M \quad (n \ge 2)
\end{equation}
where $\alpha_2=1$, $\alpha_3=2^{1/2}$, $\alpha_4=3 \cdot 2^{1/2}$, $\alpha_5=15 \cdot 2^{1/2}$ and $\alpha_n=2^{\frac{n^2+2n-18}{4}}$ for $n \ge 6$.
\end{thm}

Also, the example
\begin{equation}\label{the_example}
Q(x) := \prod_{j=1}^{n/2}(x^2-r_j^2) \qquad (2 \mid n)
\end{equation}
shows that Theorem \ref{thm:1} does not generalize directly to non-consecutive coefficients.

\section{Preliminary Lemmas}

\begin{lem}\label{lem_1}
Let $Q(x) \in \mathcal{F}_n$ be a polynomial. Then
\begin{equation}\label{ineq_1}
a_{j}^2 \ge a_{j-1}a_{j+1} \quad (1 \le j \le n-1)
\end{equation}
and
\begin{equation}\label{ineq_2}
4(a_{j}^2-a_{j-1}a_{j+1})(a_{j+1}^2-a_{j}a_{j+2}) \ge (a_{j-1}a_{j+2}-a_{j}a_{j+1})^2 \quad (1 \le j \le n-2).
\end{equation}
\end{lem}

\begin{proof}
This is Proposition $5.1$ of \cite{cpn}.
\end{proof}

\begin{lem}\label{lem_2}
Let $Q(x) \in \mathcal{F}_n$ be a polynomial. Then $Q(x)$ does not have two consecutive zero coefficients.
\end{lem}

\begin{proof}
It follows from Lemma $1.5$ of \cite{cpn}.
\end{proof}

\begin{rem}
Let $Q(x) \in \mathcal{F}_n$ be a polynomial. It follows from Lemma \ref{lem_2} and inequality \eqref{ineq_1} that if $a_j=0$ for some $j \in \{1,\dots, n-1\}$ then $a_{j-1}a_{j+1} < 0$.
\end{rem}

\section{Proof of Theorem \ref{thm:1}}

Let $Q(x) \in \mathcal{F}_n$ be a polynomial satisfying $\max(|a_{j}|,|a_{j+1}|) \le M$ for some $j \in \{1,\dots,n-1\}$.

{\bf First principle.} Clearly $\max(|x_1|,\dots,|x_n|) \le |a_n|$. Also, from
$$
a^2_1-2a_2 = x^2_1 + \cdots + x^2_n
$$
we deduce that $\max(|x_1|,\dots,|x_n|) \le \sqrt{2}\max(|a_1|,|a_2|) =: \sqrt{2}A$. Indeed, $|a^2_1-2a_2| \le A^2(1+\frac{2}{A}) \le 2A^2$ since $A=1$ implies $|a^2_1-2a_2| \le 3$ and $\max(|x_1|,\dots,|x_n|) \le 1 \le \sqrt{2}A$.\\

\noindent{\bf  Second principle.} We have
$$
\min(|a_{l-1}|,|a_{l+2}|) \le 3\max(|a_{l}|,|a_{l+1}|) \quad (1 \le l \le n-2).
$$
Indeed, we write $A:=\max(|a_{l}|,|a_{l+1}|)$. If $\min(|a_{l-1}|,|a_{l+2}|) \ge \beta A$ then $|a_{l-1}a_{l+2}-a_{l}a_{l+1}| \ge \bigl(1-\frac{1}{\beta^2}\bigr)|a_{l-1}a_{l+2}|$. From inequality \eqref{ineq_2} we find
$$
\bigl(1-\frac{1}{\beta^2}\bigr)^2|a_{l-1}a_{l+2}|^2 \le 4A^2\bigl(1+\frac{1}{\beta}\bigr)^2|a_{l-1}a_{l+2}|
$$
so that
$$
\beta A \le \min(|a_{l-1}|,|a_{l+2}|) \le \frac{2\beta}{\beta-1}A
$$
and the result follows. The example $a_{l-1}=-3, a_{l}=1, a_{l+1}=1$ and $a_{l+2}=-3$ shows that it is the best possible result that can be shown from inequality \eqref{ineq_2}.

For each $m \ge 1$ we consider the polynomials
$$
f_m(x):=\left\{\begin{array}{ll} (x+1)(x^2-1)^{\frac{m-1}{2}} & \mbox{if}\ 2 \nmid m,\\ [1mm]
(x^2-1)^{\frac{m}{2}} & \mbox{if}\ 2 \mid m. \end{array}\right.
$$
These unitary polynomials have relatively small coefficients and have only integer roots. In particular, the sum of the absolute value of the coefficients of $f_{m}(x)$ is $2^{\lfloor\frac{m+1}{2}\rfloor}$ and $|f_m(0)|=1$.

We are now ready for the proof. We will establish an upper bound of the shape $k(t)M$ for the absolute value of each coefficients in some chain of $t$ consecutive coefficients with $t \ge 2$. By hypothesis $k(2) =1$. The second principle implies that $k(3) \le 3$.

Let us assume that $k(t)$ is known. We consider the polynomial
$$
Q_t(x):=Q(x)f_{t-2}(x).
$$
Clearly, $Q_t(x)  \in \mathcal{F}_{n+t-2}$ and has two consecutive coefficients that are bounded by $2^{\lfloor\frac{t-1}{2}\rfloor}k(t)M$ in absolute value. By using the second principle, we deduce that a third coefficient of $Q_t(x)$ (consecutive to the two others) is bounded by  $3 \cdot 2^{\lfloor\frac{t-1}{2}\rfloor}k(t)M$ in absolute value. From this coefficient, we deduce that $k(t+1) \le \bigl(4 \cdot 2^{\lfloor\frac{t-1}{2}\rfloor}-1\bigr)k(t)$. One can use this idea to show that $k(4) \le 15$ directly.

We verify that
\begin{eqnarray*}
k(t) & \le & k(4)\prod_{j=2}^{t-3}\bigl(4 \cdot 2^{\lfloor\frac{j+1}{2}\rfloor}-1\bigr)\\
& \le & 2^{\frac{(t+1)^2+2(t+1)-20}{4}}
\end{eqnarray*}
for $t \ge 5$. The last inequality is shown by induction on odd and even values of $t$ separately.

The result follows from the first principle applied to a chain of length $t$ for some $t \in \{2,\dots,\max(2,n-1)\}$.

\section{Concluding remarks}

We can show that the number of binary chains of length $n$ that contain either two consecutive 1 or that end with 1 is of
$$
2^n-\biggl(\frac{1}{2}+\frac{3\sqrt{5}}{10}\biggr)\biggl(\frac{1+\sqrt{5}}{2}\biggr)^{n-1}+\biggl(-\frac{1}{2}+\frac{3\sqrt{5}}{10}\biggr)\biggl(\frac{1-\sqrt{5}}{2}\biggr)^{n-1} \quad (n \ge 2).
$$
This shows that Theorem \ref{thm:1} is enough to establish that most subsets of coefficients fix finitely many polynomials of $\mathcal{F}_n$. The example \eqref{the_example} shows that in general there are $\gg 2^{n/2}$ subsets of coefficients that do not fix finitely many polynomials of $\mathcal{F}_n$ when $2 \mid n$.

{\sc D\'epartement de math\'ematiques et de statistique, Universit\'e Laval, Pavillon Alexandre-Vachon, 1045 Avenue de la M\'edecine, Qu\'ebec, QC G1V 0A6} \\
{\it E-mail address:} {\tt Patrick.Letendre.1@ulaval.ca}


\begin{thebibliography}{99}

\bibitem{pb:te} {\sc P. Borwein} et {\sc T. Erd\'elyi}, Polynomials and polynomial inequalities, Graduate Texts in Mathematics, {\bf 161}, Springer-Verlag, New York, 1995.

\bibitem{cpn} {\sc C. P. Niculescu}, A new look at Newton's inequalities, {\it J. Inequal. in Pure and Appl. Math.}, {\bf 1}(2) (2000), Art. 17.

\bibitem{gp:gs} {\sc G. P\'olya} et {\sc G. Szeg\"o}, Problems and Theorems in Analysis, vols. I and II, Springer-Verlag, Berlin, Heidelberg, 1976.



\end{thebibliography}
\end{document}